\documentclass{amsart}
\usepackage{amssymb,amsfonts,latexsym}

\newtheorem{theorem}{Theorem}[section]
\newtheorem{lemma}[theorem]{Lemma}
\newtheorem{proposition}[theorem]{Proposition}
\newtheorem{corollary}[theorem]{Corollary}

\theoremstyle{definition}
\newtheorem{definition}[theorem]{Definition}

\newtheorem{remark}[theorem]{Remark}
\newtheorem{general remarks}[theorem]{General remarks}

\newcommand{\id}{\operatorname{id}}

\newcommand{\ord}{\operatorname{ord}}

\newcommand{\ben}{\begin{enumerate}}
\newcommand{\een}{\end{enumerate}}

\begin{document}

\title[Hilbert series of PI relative free $G$-graded algebras are rational functions]
{Hilbert series of PI relative free $G$-graded algebras are
rational functions}

\author{Eli Aljadeff}
\address{Department of Mathematics, Technion-Israel Institute of
Technology, Haifa 32000, Israel}
\email{aljadeff@tx.technion.ac.il}
\author{Alexei Kanel-Belov}
\address{Department of Mathematics, Bar-Ilan University, Ramat-Gan, Israel}
\email{beloval@macs.biu.ac.il}

\date{Oct. 30, 2010}





\thanks {The first author was partially supported by the ISRAEL SCIENCE FOUNDATION
(grant No. 1283/08) and by the E.SCHAVER RESEARCH FUND. The second
author was partially supported by the ISRAEL SCIENCE FOUNDATION
(grant No. 1178/06). The second author is grateful to the Russian Fund of Fundamental
Research for supporting his visit to India in 2008 (grant  $\ RFBR 08-01-91300-IND_a$).}

\begin{abstract}

Let $G$ be a finite group, $(g_{1},\ldots,g_{r})$ an
(unordered) $r$-tuple of $G^{(r)}$ and $x_{i,g_i}$'s variables
that correspond to the $g_i$'s, $i=1,\ldots,r$. Let $F\langle
x_{1,g_1},\ldots,x_{r,g_r} \rangle$ be the corresponding free
$G$-graded algebra where $F$ is a field of zero characteristic. Here
the degree of a monomial is determined by
the product of the indices in $G$. Let $I$ be a $G$-graded
$T$-ideal of $F\langle x_{1,g_1},\ldots,x_{r,g_r} \rangle$ which is PI
(e.g. any ideal of identities of a $G$-graded finite dimensional algebra is of this type).
We prove that the Hilbert series of $F\langle
x_{1,g_1},\ldots,x_{r,g_r} \rangle/I$ is a rational function. More
generally, we show that the Hilbert series which corresponds to
any $g$-homogeneous component of $F\langle x_{1,g_1},\ldots,x_{r,g_r}
\rangle/I$ is a rational function.

\end{abstract}

\maketitle

\bigskip
\noindent

\bigskip\bigskip

\hspace{3cm}

\begin{section}{Introduction} \label{Introduction}

The Hilbert series of an affine (i.e. finitely generated) algebra
and its computation, is a topic which attracted a lot of attention
in the last century , classically in commutative algebra (see e.g.
\cite{Stanley1}, \cite{Stanley2}), but also (and in fact more
importantly for the purpose of this paper) in non commutative
algebra (see e.g. \cite{Anick}). In particular the question of
when the Hilbert series $H_{W}$ of and algebra $W$ is the Taylor
expansion of a rational function is fundamental in the theory, and
when it does, it serves as a ``fine" structural invariant of $W$
with applications to growth invariants (see  \cite{Dren0},
\cite{Dren1}, \cite{Dren2}, \cite{Kob}, \cite{Koshlu},
\cite{Shearer}).

In case the algebra $W$ is a
relatively free algebra, i.e. isomorphic to the quotient of an
affine free algebra $F \langle x_1, \ldots, x_n \rangle$ by a
$T$-ideal of identities $I$ it is known that $H_{F \langle x_1,
\ldots, x_n \rangle/I}$ is a rational function (see \cite{Belov}, \cite{BR}) and this
fact has been successfully used in the estimation of the
asymptotic behavior of the co-character sequence of a PI algebra.

Specifically in \cite{BereleRegev} (based on explicit formulas for
the co-character sequences which appear in \cite{BER}) the authors
show that if $A$ is PI-algebra with $1$ which satisfies a Capelli
identity, then the asymptotic behavior of the codimension sequence
is of the form

$$
c_{n}(A)=an^{g}l^{n}
$$
where $a$ is a scalar, $2g$ is an integer and $l$ is a non
negative integer (we refer the reader to \cite{gz1}, \cite{gz2}
and \cite{gz3} for a comprehensive account on the codimension
sequence of a PI algebra).

The rationality of the Hilbert series of an affine relatively free
algebra has been established also in case $W$ is a super algebra
(i.e. $Z_{2}$-graded) and our goal in this paper is to extend
these results to $G$-graded relatively free affine PI algebras
where $G$ is an arbitrary finite group.

Let $\Omega=(g_1,...,g_r)$ be an unordered $r$-tuple of elements
of $G$ (in particular we allow repetitions). Let
$X_{G}=\{x_{(1,g_1)},...,x_{(r,g_r)}\}$ be a set of variables
which correspond to the set $\Omega$ and let $F\langle
X_{G}\rangle$ be the free algebra generated by $X_{G}$ over $F$.
In order to keep the notation as light as possible we may omit one
index and write $X_{G}=\{x_{g_1},...,x_{g_r}\}$ with the
convention that the variables $x_{g_i}$ and $x_{g_j}$ are
different even if $g_i=g_j$ in $G$. We equip $F\langle
X_{G}\rangle$ with a (natural) $G$-grading, namely a monomial of
the form $x_{g_{s_1}}x_{g_{s_2}}\cdots x_{g_{s_m}}$ is homogeneous
of degree $g_{s_1}~ g_{s_2}\cdots g_{s_m}\in G$. We also consider
its homogeneous degree $m\in \mathbb{N}$, namely the number of
variables in the monomial. Let $I$ be a $G$-graded $T$-ideal, that
is, $I$ is closed under $G$-graded endomorphisms of $F\langle
X_{G}\rangle$. This implies (by a standard argument) that $I$ is
generated by strongly $G$-graded polynomials, i.e. $G$-graded
polynomials whose different monomials are permutations of each
other and moreover they all have the same $G$-degree (see
\cite{AHN}).

We will assume in addition that the $T$-ideal $I$ contains an
ordinary polynomial. By this we mean that $I$ contains all
$G$-graded polynomials obtained by assigning all possible degrees
in $G$ to the variables $x_i$ of an ordinary non zero polynomial
$p(x_1,\ldots,x_n)$. We refer to a $G$-graded $T$-ideal which
contains an ordinary polynomial as a PI $G$-graded $T$-ideal.

\begin{remark}

Note for instance, that the $T$-ideal of $G$-graded identities of any
finite dimensional $G$-graded algebra is PI. On the other hand if $G \neq  \{e\}$, then the $G$-graded algebra $W$ where $W_{e}$
is a free noncommutative algebra and $W_{g}=0$ for $g \neq e$ is $G$-graded PI but of course not PI.

\end{remark}

Let $F\langle x_{g_1},\ldots,x_{g_r} \rangle$ be the free
$G$-graded algebra generated by homogeneous variables
$\{x_{g_i}\}_{i=1}^{r}$. As above we let $I$ be a $G$-graded $T$-ideal which
is PI and let $\Phi=F\langle x_{g_1},\ldots,x_{g_r} \rangle/I$ be the
corresponding relatively free $G$-graded algebra. Let $\Omega_{n}$ be the (finite) set of monomials of degree $n$ on the
$x_{g}$'s and let $c_{n}$ be the dimension of the
$F$-subspace of $\Phi$
spanned by the monomials of $\Omega_{n}$. We denote by

$$
H_{F\langle x_{g_1},\ldots,x_{g_r} \rangle/I}(t)=
\sum_{n}c_{n}t^{n}
$$
the Hilbert series of $\Phi$ with respect to the $x_{g}$'s.

\begin{theorem} \label{main basic theorem}

The series $H_{F\langle x_{g_1},\ldots,x_{g_r} \rangle/I}(t)$ is the Taylor series
of a rational function.

\end{theorem}

In our proof we strongly use key ingredients which appear in the
proof of representability of the $G$-graded relatively free
algebras (see \cite{AB}). These ingredients include the existence
of certain polynomials, called Kemer polynomials, which are
``extremal non identities'' (see \cite{AB}) and their existence
relies on the fundamental fact that the Jacobson radical of a PI
algebra is nilpotent (see \cite{Amitsur}, \cite{Braun_nilpotent},
\cite{Kemer_Hilbert} and \cite{Rasmy-Nilpotency}) (this parallels
the Hilbert's Nullstellensatz in the commutative theory), and the
solution of the Specht problem (see \cite{AB}, \cite{Kemer_rep_0},
\cite{Kemer_rep_1} and \cite{Sv}) (which parallels the Hilbert
basis theorem in the commutative theory). We emphasize however
that the rationality of the Hilbert series is not a corollary of
representability since there are examples of representable
algebras which have a transcendental Hilbert series. Indeed,
recall from \cite{BR} that the monomial algebra supported by
monomials of the form $ h=x_{1}x_{2}^{m}x_{3}x_{4}^{n}x_{5} $ with
the extra relation that $h=0$ if $m^{2}-2n^{2}=1$ is representable
but with a transcendental Hilbert series. For more on this we
refer the reader to \cite{Belov443} and \cite{BelovRowenShirshov}.

In Theorems \ref{main multivariate} and \ref{main single comp}
below we present different generalizations of Theorem \ref{main
basic theorem}. Then, in Theorem \ref{main full}, these are
combined into one general statement. We start with multivariate
Hilbert series.

Let $F\langle x_{g_1},\ldots,x_{g_r} \rangle$ be the free
$G$-graded algebra generated by homogeneous variables
$\{x_{g_i}\}_{i=1}^{r}$. As above we let $I$ be a $G$-graded $T$-ideal which
is PI and let $F\langle x_{g_1},\ldots,x_{g_r} \rangle/I$ be the
corresponding relatively free $G$-graded algebra. For any
$r$-tuple of non-negative integers $(d_1,\ldots,d_r)$ we consider
the (finite) set of monomials $\Omega_{(d_1,\ldots,d_r)}$ on
$x_g$'s where the variable $x_{g_i}$ appears exactly $d_{i}$
times, $i=1,\ldots,r$. We denote by $c_{(d_1,\ldots,d_r)}$ the dimension of the
$F$-subspace of $F\langle x_{g_1},\ldots,x_{g_r} \rangle/I$
spanned by the monomials in $\Omega_{(d_1,\ldots,d_r)}$.
\begin{remark}

As mentioned above the $T$ ideal $I$ is generated by multilinear
$G$-graded polynomials which are strongly homogeneous i.e. its monomials have the same
homogeneous degree in $G$. This fact will
play an important role in the sequel.
\end{remark}

\begin{definition}

Notation as above. The multivariate Hilbert series of $F\langle x_{g_1},\ldots,x_{g_r} \rangle/I$ is given by

$H_{F\langle x_{g_1},\ldots,x_{g_r} \rangle/I}(t_1,\ldots,t_r)=
\sum_{(d_1,\ldots,d_r)}c_{(d_1,\ldots,d_r)}t_{1}^{d_1}\cdots
t_{r}^{d_r}.$

\end{definition}

The following result generalizes Theorem \ref{main basic theorem}.

\begin{theorem} \label{main multivariate}

Notation as above. The Hilbert series $H_{F\langle
x_{g_1},\ldots,x_{g_r} \rangle/I}(t_1,\ldots,t_r)$ is a rational
function.

\end{theorem}

\begin{remark}
Clearly, Theorem \ref{main basic theorem} follows from Theorem
\ref{main multivariate} simply by replacing all variables $t_i$ by
a single variable $t$.
\end{remark}

The next generalization of Theorem \ref{main basic theorem} is in
a different direction. We consider the Hilbert series (in one
variable $t$) of a unique homogeneous component. More precisely we
fix $g \in G$ and we consider the set of monomials $\Omega_{g,n}$
of degree $n$ whose homogeneous degree is $g$. We let $c_{g,n}$ be
the dimension of the subspace in $F\langle x_{g_1},\ldots,x_{g_r}
\rangle/I$ spanned by the monomials in $\Omega_{g,n}$ (or rather,
by the elements in $F\langle x_{g_1},\ldots,x_{g_r} \rangle/I$
they represent).

\begin{definition}

With the above notation, the Hilbert series of the $g$-component
of $F\langle x_{g_1},\ldots,x_{g_r} \rangle/I$ is given by

$$
H_{g, F\langle x_{g_1},\ldots,x_{g_r} \rangle/I}(t)=\sum c_{g,n}t^{n}
$$

\end{definition}

\begin{theorem} \label{main single comp}

With the above notation. The Hilbert series $ H_{g, F\langle
x_{g_1},\ldots,x_{g_r} \rangle/I}(t)$ is a rational function.

\end{theorem}

The case where $g=e$ is of particular interest. Moreover we could
consider the Hilbert series of any collection of $g$-homogeneous
components and in particular the Hilbert series which
corresponds to a subgroup $H$ of $G$.

Finally we combine the generalizations which appeared in Theorems
\ref{main multivariate} and \ref{main single comp}. Fix an element
$g \in G$. For any $r$-tuple $(d_1,...,d_r)$ of non-negative
integers we consider the monomials with $x_{g_1}$ appearing $d_1$
times, $x_{g_2}$ appearing $d_2$ times, $\ldots$, whose
homogeneous degree is $g \in G$. We denote this set of monomials
by $\Omega_{g,(d_1,...,d_r)}$ and we let $c_{g,(d_1,...,d_r)}$ be
the dimension of the space in $F\langle x_{g_1},\ldots,x_{g_r}
\rangle/I$ spanned by elements whose representatives are the
monomials in $\Omega_{g,(d_1,...,d_r)}$.

\begin{theorem} \label{main full}

The Hilbert series

$$
H_{g,F\langle x_{g_1},\ldots,x_{g_r} \rangle/I}(t_1,\ldots,t_r) =
\sum c_{g,(d_1,...,d_r)}t_{1}^{d_1}\cdots t_{r}^{d_r}
$$
which corresponds to the $g$-component of the $G$-graded algebra
$F\langle x_{g_1},\ldots,x_{g_r} \rangle/I$ is a rational
function.

\end{theorem}

\begin{remark}

Clearly, Theorems \ref{main basic theorem}, \ref{main
multivariate}, \ref{main single comp} are direct corollaries of
Theorem \ref{main full}.

\end{remark}

Theorem \ref{main full} is proved in Section \ref{Prelim-PROOF}.
As mentioned above the proof uses ingredients from the proof of
the representability of relatively free $G$-graded, affine, PI
algebras. For the reader convenience we recall in the first part
of the section the required results from \cite{AB} which are used
in the proof.

In Section \ref{SPECIAL CASE} we consider the $T$-ideal of
$G$-graded identities $I$ of the group algebra $FG$ and show, in a
rather direct way, the rationality of the Hilbert series of the
corresponding relatively free algebra $F\langle
x_{g_1},\ldots,x_{g_r} \rangle/I$. We also give in this case (see
Corollary \ref{codimension for FG}) the asymptotic behavior of the
corresponding codimension sequence. Furthermore, these results can
be easily extended to twisted group algebras $F^{\alpha}G$, where
$\alpha \in Z^{2}(G,F^{*})$ (see Remark \ref{twisting}). The
$G$-grading determined by twisted group algebras is called ``fine"
and it plays an important role in the classification of $G$-graded
simple algebras (see \cite{BSZ}).

We close the introduction by mentioning that one may consider the
Hilbert series of an affine $G$-graded free algebra. It is easy to
see that the corresponding Hilbert series is a rational function.
Of course one may ask whether the different generalizations (Theorems \ref{main multivariate}
and \ref{main single comp}) apply also in this case (i.e. when
$I=0$). It is easy to see that this is indeed the case when
extending to multivariate series. As for the second
generalization, we refer the reader to \cite{FM}. The authors
prove the rationality of the Hilbert series which corresponds to
the sub algebra of coinvariant elements, that is, in case $g=e$.

\end{section}

\begin{section}{Preliminaries and proofs} \label{Prelim-PROOF}

We start this section by recalling some facts on $G$-graded
algebras $W$ over a field of characteristic zero $F$ and their
corresponding $G$-graded identities. We refer the reader to
\cite{AB} for a detailed account on this topic.

Let $W$ be an affine $G$-graded PI algebra over $F$. We denote by
$I=\id_{G}(W)$ the ideal of $G$-graded identities of $W$. These
are polynomials in the free $G$-graded algebra over $F$ which are
generated by $X_{G}$ and that vanish upon any admissible
evaluation on $W$. Here $X_{G}=\bigcup X_{g}$ and $X_{g}$ is a set
of countably many variables of degree $g$. An evaluation is
admissible if the variables from $X_{g}$ are replaced only by
elements from $W_{g}$. It is known that $I$ is a $G$-graded
$T$-ideal, i.e. closed under $G$-graded endomorphisms of $F
\langle X_{G} \rangle$.

We recall from \cite{AB} that the $T$-ideal $I$ is generated by
multilinear polynomials. Consequently, the $T$-ideal of identities
does not change when passing to $\overline{F}$, the algebraic
closure of $F$, in the sense that the ideal of identities of
$W_{\overline{F}}$ over $\overline{F}$ is the span (over
$\overline{F}$) of the $T$-ideal of identities of $W$ over $F$. It
is easily checked that the Hilbert series remains the same when
passing to the algebraic closure of $F$. Thus, from now on we
assume $F=\overline{F}$ .

Next, we recall some terminology and some facts from Kemer theory
extended to the context of $G$-graded algebras as they appear in
\cite{AB}. We start with the concept of alternating polynomial on
a set of variables.

Let $f(x_{1,g},\ldots,x_{r,g};y_{1},\ldots, y_{n})$ be a
multilinear polynomial with variables $(x_{1,g},\ldots,x_{r,g})$,
homogeneous of degree $g$, and some other variables $y$'s,
homogeneous of unspecified degrees. We say that $f$ is alternating
on the set $x_{1,g},\ldots,x_{r,g}$ if there is a multilinear
polynomial $h(x_{1,g},\ldots,x_{r,g};y_{1},\ldots, y_{n})$ such
that

$$
f(x_{1,g},\ldots,x_{r,g};y_{1},\ldots, y_{n})=\sum_{\sigma \in
Sym(r)}(-1)^{\sigma}h(x_{\sigma(1),g},\ldots,x_{\sigma(r),g};y_{1},\ldots,
y_{n}).
$$

We say that a polynomial $f$ alternates on a collection of
disjoint sets of homogeneous variables (each set constitute of
variables of the same degree), if it is alternating on each set.

In the sequel we will need to consider multilinear polynomials $f$
which alternates on $d$ disjoint sets of $g$-elements, each of
cardinality $r$. More generally, we will consider multilinear
polynomials such that for any $g \in G$, $f$ contains $n_{g}$
disjoint sets of variables of homogeneous degree $g$ and each set
of cardinality $d_{g}$.

We recall from \cite{AB} that a $G$-graded polynomial with an
alternating sets of $g$-homogeneous variables which is ``large
enough" is necessarily an identity. More precisely, for any affine
PI $G$-graded algebra $W$ and for any $g \in G$ there exists an
integer $d_g$ such that any $G$-graded polynomial which has an
alternating set of $g$ variables of cardinality exceeding $d_g$ is
necessarily a $G$-graded identity of $W$.

In particular this holds for a finite dimensional $G$-graded
algebra $A$. Note that in that case, if  a polynomial $f$ has an alternating set of
$g$-homogeneous elements whose cardinality exceeds the dimension
of $A_g$, it is clearly an identity of $A$.

Next we recall that by Wedderburn-Malcev decomposition theorem, a $G$-graded
finite dimensional algebra $A$ over $F$, may be decomposed into the direct
sum of $\overline{A}\oplus J$ (decomposition as vector spaces)
where $J$ is the Jacobson radical ($G$-graded) and $\overline{A}$ is a (semisimple) subalgebra of $A$ isomorphic to
$A/J$ as $G$-graded algebras. As a consequence we have
decompositions of $\overline{A}$ and $J$ to the corresponding
$g$-homogeneous components.

Now, we know that in order to test whether a multilinear
polynomial is an identity of an algebra, it is sufficient to
evaluate its variables on a base and hence, applying the above
decomposition, we may consider semisimple and radical $G$-graded
evaluations. From these considerations we conclude that if a
polynomial has sufficiently many alternating sets of
$g$-homogeneous elements of cardinality that exceeds the dimension
of the $g$-homogeneous component of $\overline{A}$, the polynomial
is necessary an identity of $A$. Indeed, in any evaluation we
either have a semisimple basis element which appears twice or at
least one of the evaluations is radical. In the first case we get
zero as a result of the alternation of two elements which are
equal, while in the second, we get zero if the number of
alternating sets is at least the nilpotency index of $J$. We
therefore see that if a $G$-graded polynomial $f$ has a number of
alternating sets (same cardinality) of $g$-variables which is at
least the nilpotency index of $J$ (and in particular if it has
``sufficiently many") then the cardinality of the sets must be
bounded by the dimension of $\overline{A}_{g}$ if we know that $f$
is a non ($G$-graded) identity of $A$. Thus, if $f$ is a non
($G$-graded) identity of $A$ with $\alpha_{g}$ disjoint
alternating sets of $g$-homogeneous elements of cardinality
$d_g+1$, $g \in G$ where $d_{g}=\dim(\overline{A}_{g})$ then, $
\sum_{g} \alpha_{g} \leq n-1$, where $n$ is the nilpotency index
of $J$.

In \cite{AB} the notion of a finite dimensional $G$-graded
\textit{basic} algebra was introduced. For our exposition here, it
is not necessary to recall its precise definition but only say (as
a result of Kemer's Lemmas 1 and 2 for $G$-graded algebras (see
sections $5$ and $6$ in \cite{AB})) that any basic algebra $A$
admits non-identities $G$-graded polynomials which have
``arbitrary many" (say at least $n$, the nilpotency index of $J$)
alternating sets of $g$-homogeneous variables of cardinality $d_g$
and precisely (a total of) $n-1$ alternating sets of
$g$-homogeneous variables of cardinality $d_g+1$ for some $g \in
G$. These extremal non identities are the so called ``$G$-graded
Kemer polynomials" of the basic algebra $A$.

Thus to each basic algebra $A$ corresponds an $r+1$-tuple of
non-negative integers $(d_{g_1},\ldots,d_{g_r};n-1)$ where $d_{g}$
is the dimension of the $g$-component of the semisimple part of
$A$ and $n$ is the nilpotency index of $J$, the radical of $A$. We
refer to such a tuple as the Kemer point of the basic algebra $A$.

The representability theorem for affine $G$-graded algebras can be stated as follows.

Given an affine PI $G$-graded algebra $W$ over $F$, where $F$ is a
field of zero characteristic, there exists a finite number of
$G$-graded basic algebras $A_1,\ldots,A_m$ over a field extension
$K$ of $F$ such that $W$ satisfies the same $G$-graded identities
as $A_1\oplus \cdots \oplus A_m$. Note that since
$\id_{G}(A_1\oplus \cdots \oplus A_n)=\bigcap \id_{G}(A_i)$ we may
assume $\id(A_i)\nsubseteq \id(A_j)$ for every $1 \leq i,j \leq
m$.

\begin{remark}

In fact, by passing to the algebraic closure of $K$, we may assume
that the algebras $A_{i}$ above are finite dimensional over the
same field $F$.

\end{remark}

\begin{definition}

With the above notation, we say that a finite dimensional
$G$-graded algebra $A$ is subdirectly irreducible if it has no non
trivial, two sided $G$-graded ideals $I$ and $J$ such that $I\cap
J = (0)$.

\end{definition}

\begin{remark}

Note that if the algebra $A_1\oplus \cdots \oplus A_m$ is subdirectly irreducible then $m=1$.

\end{remark}

A key ingredient in the proof of Theorem \ref{main full} is the
existence of an essential Shirshov base for the relatively free
algebra $F\langle x_{g_1},\ldots,x_{g_r} \rangle/I$.

For the reader convenience we recall the necessary definitions and
statements from \cite{AB}, starting from the ordinary case (i.e.
ungraded).

\begin{definition}

Let $W$ be an affine PI-algebra over $F$. Let
$\{a_1,\ldots,a_s\}$ be a set of generators of $W$. Let $m$ be a
positive integer and let $Y$ be the set of all words in
$\{a_1,\ldots,a_s\}$ of length $\leq m$. We say that $W$ has
\textit{Shirshov} base of length $m$ and of height $h$ if elements
of the form $y_{i_1}^{k_1}\cdots y_{i_l}^{k_l}$ where $y_{i_i} \in
Y$ and $l\leq h$, span $W$ as a vector space over $F$.

\end{definition}

\begin{theorem}

If $W$ is an affine PI-algebra, then it has a Shirshov
base for some $m$ and $h$. More precisely, suppose $W$ is
generated by a set of elements of cardinality $s$ and suppose it
has PI-degree $m$ (i.e. there exists an identity of
degree $m$ and $m$ is minimal) then $W$ has a Shirshov base of
length $m$ and of height $h$ where $h=h(m,s)$.

\end{theorem}

In fact we will need a weaker condition (see \cite{AB})

\begin{definition}

Let $W$ be an affine PI-algebra. We say that a set $Y$ as
above is an \textit{essential Shirshov} base of $W$  (of length
$m$ and of height $h$) if there exists a finite set $D(W)$ such
that the elements of the form $d_{i_1}y_{i_1}^{k_1}d_{i_2} \cdots
d_{i_l}y_{i_l}^{k_l}d_{i_{l+1}}$ where $d_{i_j}\in D(W)$,
$y_{i_j}\in Y$ and $l\leq h$ span $W$.
\end{definition}

An essential Shirshov's base gives the following.

\begin{theorem}\label{finite module}

Let $C$ be a commutative ring and let $W=C\langle
\{a_1,\ldots,a_s\}\rangle$ be an affine algebra over $C$. If $W$
has an essential Shirshov base (in particular, if $W$ has a
Shirshov base) whose elements are integral over $C$, then it is a
finite module over $C$.
\end{theorem}

Returning to $G$ graded algebras we have the following.

\begin{proposition}\label{Essential Shirshov base in the
e-component}

Let $W$ be an affine, PI, $G$-graded algebra. Then it has
an essential $G$-graded Shirshov base of elements of $W_{e}$.

\end{proposition}

For the proof of Theorem \ref{main full} we will assume below
there exists a $G$-graded $T$-ideal $I$ which is non-rational
(that is the Hilbert series of $F\langle x_{g_1},\ldots,x_{g_r}
\rangle/I$ is non-rational) and get a contradiction. The next two
lemmas will be used to reduce the problem to the case where $I$ is
maximal with respect to being non-rational and also that the
relatively free algebra $F\langle x_{g_1},\ldots,x_{g_r}
\rangle/I$ is $G$-graded PI equivalent to one basic algebra
(rather than to a direct sum of them).

Before stating the lemmas we simplify the terminology as follows.

\begin{remark}
If $X$ is a subspace of an algebra $U=F\langle x_{g_1},...,x_{g_r}
\rangle/I$ spanned by strongly homogeneous polynomials, then we
may consider naturally its corresponding multivariate Hilbert
series. Then when we say ``Hilbert series of $X$" we mean
``multivariate $G$-graded Hilbert series which corresponds to the
$g$-component of $X$" for any given $g \in G$.

\end{remark}

\begin{lemma} \label{FIRST REDUCTION}

Let $J$ be a $G$-graded $T$-ideal containing $I$. Let $H_U$, $H_{U/J}$ and $H_{J/I}$ be the Hilbert series of
$U$, $U/J$ and $J/I$ respectively.  Then

$$
H_{F\langle x_{g_1},...,x_{g_r} \rangle/I}=H_{F\langle x_{g_1},...,x_{g_r}
\rangle/J}+H_{J/I}.
$$

\end{lemma}

\begin{proof}

This is clear since $J$ is spanned by strongly homogeneous polynomials.

\end{proof}

\begin{lemma} \label{SECOND REDUCTION}

Let $I^{'}$ and $I{''}$ be two $G$-graded $T$-ideals which contain
$I$. Then the following holds:

$$
H_{F\langle x_{g_1},...,x_{g_r} \rangle/(I^{'}\cap I^{''})} = H_{F\langle
x_{g_1},...,x_{g_r} \rangle/I^{'}} + H_{F\langle x_{g_1},...,x_{g_r}
\rangle/I^{''}} - H_{F\langle x_{g_1},...,x_{g_r} \rangle/(I^{'}+ I^{''})}
$$

\end{lemma}

\begin{proof}
The proof is similar to the proof of Lemma 9.40 in \cite{BR} and
hence is omitted.
\end{proof}

Let us assume now that there exist $G$-graded $T$-ideals of
$F\langle x_{g_1},...,x_{g_r}\rangle$ which are PI and such that
the Hilbert series of $F\langle x_{g_1},...,x_{g_r}\rangle/I$ is
non-rational.

\begin{proposition} \label{MAXIMAL AND IRREDUCIBLE}

Under the above assumption there exists a $G$-graded $T$-ideal of
$F\langle x_{g_1},...,x_{g_r}\rangle$ which is PI and

\begin{enumerate}

\item

is maximal with respect to the property that the Hilbert series of
$F\langle x_{g_1},...,x_{g_r}\rangle/I$ is non-rational.

\item

the relatively free algebra $F\langle x_{g_1},...,x_{g_r}\rangle/I$ is $G$-graded PI equivalent to
a $G$-graded basic algebra $A$.

\end{enumerate}

\end{proposition}

\begin{proof}

The first assertion follows from the $G$-graded Specht property
(see section $12$ in \cite{AB}). Indeed, if there is no such an
ideal then we get an infinite ascending sequence of ideals which
does not stabilize and this contradicts the fact that the union of
the $T$-ideals is finitely generated. We thus may assume that the
Hilbert series of $F\langle x_{g_1},...,x_{g_r}\rangle/I$ is
non-rational and the Hilbert series of $F\langle
x_{g_1},...,x_{g_r}\rangle/J$ for any $G$-graded $T$-ideal $J$
which properly contains $I$ is a rational function.

For the proof of the second assertion let us show that the
maximality of $I$ already implies that $F\langle
x_{g_1},...,x_{g_r}\rangle/I$ is PI equivalent to a $G$-graded
basic algebra. Assuming the converse, we have that $F\langle
x_{g_1},...,x_{g_r}\rangle/I$ is $G$-graded PI equivalent to a
direct sum of basic algebras $A_1\oplus A_2 \oplus \ldots \oplus
A_m$ where $m\geq 2$ and $\id_{G}(A_i)\nsubseteq \id_{G}(A_j)$ for
any $1\leq i,j \leq m$. It follows that the ideal
$\id_{G}(F\langle x_{g_1},...,x_{g_r}\rangle/I)=\id_{G}(A_1\oplus
A_2 \oplus \ldots \oplus A_m)=\bigcap \id_{G}(A_i)$ and
$\id_{G}(F\langle x_{g_1},...,x_{g_r}\rangle/I) \varsubsetneq
\id_{G}(A_i)$ for any $i$. Now, consider the evaluations $I_{i}$
of the $T$-ideals $\id_{G}(A_i)$ on $F\langle
x_{g_1},...,x_{g_r}\rangle$. Clearly, $I_i$ properly contains $I$
and their intersection is $I$. By Lemma \ref{SECOND REDUCTION} we
conclude that the Hilbert series of $F\langle
x_{g_1},...,x_{g_r}\rangle/I$ is rational. Contradiction.

\end{proof}

It is convenient to view the relatively free algebra $F\langle
x_{g_1},...,x_{g_r}\rangle/I$ as an algebra of generic elements,
$G$-graded embedded in a matrix algebra over a suitable rational
function field over $F$. Indeed, by part (2) of Proposition
\ref{MAXIMAL AND IRREDUCIBLE} we have that $F\langle
x_{g_1},...,x_{g_r}\rangle/I$ is $G$-graded PI equivalent to a
$G$-graded basic algebra which we denote by $A$ and which from now
on will be viewed a $G$-graded subalgebra of the $n \times n$
matrices over $F$ (see \cite{AB}). If
$\{v_{g,1},\ldots,v_{g,s_g}\}$ is an $F$-basis of $A_{g}$, the
$g$-homogeneous component of $A$, we consider different sets of
central indeterminates $\{t_{g,1},\ldots,t_{g,s_g}\}$, $g \in G$
(one set for each generator $x_g$ of the free algebra).

Let $K=F\langle \{t_{g,i}\} \rangle$ be the field of rational
functions on the $t$'s and $A_{K}$ be the algebra over $K$
obtained from $A$ by extending scalars from $F$ to $K$. The
algebra of generic elements will be an $F$-subalgebra of $A_{K}$.
For each variable $x_{g}=x_{g_j}$ in the generating set of
$F\langle x_{g_1},...,x_{g_r}\rangle$ we form an element
$z_{g}=\sum_{i}t_{g,i}v_{g,i}$ in $A_{K}$. Following the embedding
of $A$ in $M_{n}(F)$ we view the generating elements $z_{g}$'s as
$n\times n$-matrices over the field $K$. Observe that these
generating elements are matrices whose entries are homogeneous
polynomials (on the $t$'s) of degree one.

\begin{proposition} [see \cite{AB}]

There is an $F$-isomorphism of $G$-graded algebras of $F\langle
x_{g_1},...,x_{g_r}\rangle/I$ with $\mathcal{A}$, the $F$-sub
algebra of $A_{K}$ generated by the elements
$\{z_{g_1},\ldots,z_{g_r}\}$.

\end{proposition}

\begin{remark}
In the rest of the proof we use without further notice the
identification of the relatively free algebra with the algebra of
generic elements $\mathcal{A}$ embedded in $M_{n}(K)$.

\end{remark}

Now we recall from Proposition \ref{Essential Shirshov base in the
e-component} that the relatively free algebra $F\langle
x_{g_1},...,x_{g_r}\rangle/I$ has an essential Shirshov $\Theta$
base which is contained in the $e$-component. Picking the natural
set of generators of $F\langle x_{g_1},...,x_{g_r}\rangle/I$ we
note that $\Theta$ consists of monomials on the $x_{g_i}$'s and
hence, viewed in $\mathcal{A}$, they consist of matrices over $K$
whose entries are homogeneous polynomial in the $t_i$'s. It then
follows that the characteristic values of the elements in $\Theta$
are homogeneous polynomials on the $t_i$'s.

Denote by $C$ the algebra over $F$ generated by these homogeneous
polynomials. Note that $C$ is an affine commutative algebra (the
essential Shirshov base is finite). Moreover, if we extend the
algebra of generic elements $\mathcal{A}$ to $C$ we have (by
Theorem \ref{finite module}) that $\mathcal{A_{C}}$ is a finite
module over $C$.

For $\mathcal{A_{C}}$ we know the following result.

\begin{lemma} \label{RATIONALITY FOR FINITE MODULES}

The Hilbert series of $\mathcal{A_{C}}$ is rational. More
generally, the rationality of the Hilbert series is independent of
the integer degree given to the generators.

\end{lemma}

\begin{proof}

Indeed, the algebra is a finitely generated module over an affine
domain and hence its Hilbert series is rational. (see \cite{BR},
Prop. 9.33).

\end{proof}

More generally, we may consider $C$-submodules $M$ of
$\mathcal{A_{C}}$ which are generated by strongly homogeneous
polynomials on the $x_{g_i}$'s.

\begin{lemma}

The Hilbert series of $M$ is rational.

\end{lemma}

\begin{proof}

Since $\mathcal{A_{C}}$ is a finite module over a Noetherian
domain, the module $M$ is finitely generated as well. The result
now follows from the second part of Lemma \ref{RATIONALITY FOR
FINITE MODULES}.

\end{proof}

We can complete now the proof of Theorem \ref{main full}. Let $f$
be a $G$-graded Kemer polynomial of the basic algebra $A$ and let
$J$ be the $G$-graded $T$-ideal it generates together with $I$.
Note that since $f$ is a non identity of $A$ (and hence a non
identity of $F\langle x_{g_1},...,x_{g_r}\rangle/I$) the $T$-ideal
$J$ strictly contains $I$, and hence, by the maximality of $I$ we
have that the Hilbert series of $F\langle
x_{g_1},...,x_{g_r}\rangle/J$ is rational. The key property which
we need here is that the ideal $J/I$ is closed under the
multiplication of the characteristic values (see \cite{AB}, Prop.
8.2). Hence the ideal $J/I$ of the relatively free algebra is in
fact a $C$-submodule of $\mathcal{A_{C}}$. Hence its Hilbert
series is rational. Applying Lemma \ref{FIRST REDUCTION} the
result follows. This completes the proof of Theorem \ref{main
full}.

\end{section}

\begin{section} {A special case} \label{SPECIAL CASE}

In this section we show by direct computations the rationality of
the Hilbert series of the affine relative free $G$-graded algebra
in case $I$ is the $T$-ideal of $G$-graded identities of the group
algebra $FG$. In addition we obtain a precise estimation of the asymptotic behavior
of the $G$-graded codimension sequence for that case.

Let $\overline{\alpha} = (g_1, g_2,..., g_r)$ be an $r$-tuple in
$G^{(r)}$. As in previous sections we consider the free $G$-graded
algebra $F\langle x_{1,g_1},\ldots,x_{r,g_r} \rangle$, where the
$x_{i,g_{i}}$'s are non-commuting variables which are in one to
one correspondence with the entries of $\overline{\alpha}$.
As above, we may abuse notation by deleting the index $i$ and simply write $x_{g_{i}}$.

Next we consider the $T$-ideal of $G$-graded identities
$\id_{G}(FG)$ of the group algebra $FG$. Recall from \cite{AHN}
that $\id_{G}(FG)$ is generated as a $T$-ideal by binomial
identities of the form $x_{g_{i_1}}x_{g_{i_2}}\cdots
x_{g_{i_n}}-x_{g_{i_{\sigma(1)}}}x_{g_{i_{\sigma(2)}}}\cdots
x_{g_{i_{\sigma(n)}}}$ where $\sigma$ is a permutation in $S_{n}$
and the products $g_{i_1}g_{i_2}\cdots g_{i_n}$ and
$g_{i_{\sigma(1)}}g_{i_{\sigma(2)}}\cdots g_{i_{\sigma(n)}}$
coincide in $G$. That is two monomials are equivalent if and only
if they have the same variables and they determine elements in the
same $g$-homogeneous component. In particular if the group $G$ is
abelian, then two monomials are equivalent if and only if they
have the same variables.

\begin{remark}
Clearly, by passing to a subgroup of $G$ if necessary, we may
assume the elements of $\overline{\alpha}=(g_1, g_2,..., g_r)$
generate the group $G$.
\end{remark}

Fix a natural number $n$ and consider monomials of degree
$n$ with $n_1$ variables $x_{g_1}$, $n_2$ variables $x_{g_2}$,
\ldots, $n_r$ variables $x_{g_r}$ where $n_1 + n_2 + \cdots + n_r
= n$. Now consider permutations of any monomial of that form. Clearly, any
permutation determines the same element in the abelianization of
$G$. In other words the elements in $G$ determined by two
monomials which have the same variables lie in the same coset of
the commutator $G^{'}$. In the next lemma we show that if the
monomial is ``rich enough" we may obtain all elements of a
$G^{'}$-coset. In order to state the lemma we need the following notation.

We say that the word $ \Sigma = g_{i_1}g_{i_2} \cdots g_{i_n}$ is
a presentation of $g \in G$ (in terms of the entries of
$\overline{\alpha} = (g_1, g_2,..., g_r)$) if $g=g_{i_1}g_{i_2}
\cdots g_{i_n}$ in $G$. For any word $\Sigma$ we may consider the
corresponding \textit{monomial} (in the free algebra)
$X_{\Sigma}=x_{g_{i_1}}x_{g_{i_2}}\cdots x_{g_{i_n}}$. We say that
the monomial $X_{\Sigma}$ represents $g$ in $G$.

\begin{lemma} \label{RICH WORDS}

For every $z \in G$ there exists an integer $n$ and a word $g_{i_1}g_{i_2}\cdots
g_{i_n}$ in $G$ such that the set of words in

$$
\Omega_{g_{i_1}g_{i_2}\cdots g_{i_n}} =\{g_{\sigma}=
g_{i_{\sigma(1)}}g_{i_{\sigma(2)}}\cdots g_{i_{\sigma(n)}}, \sigma
\in Sym(n) \}
$$
represent all elements of the form $zg$ where $g \in G^{'}$ (i.e. the full coset of $G^{'}$ in $G$ represented by $z$).

\end{lemma}

\begin{proof}

Clearly it is sufficient to find a word whose permutations yield
all elements of $G^{'}$. It is well known that the commutator
subgroup $G^{'}$ of $G$ is generated by commutators
$[g,h]=ghg^{-1}h^{-1}$ where $g,h \in G$. For any commutator
$[g,h]$ we write the elements $g$ and $h$ as words in the entries
of $\overline{\alpha}$ and then we write $g^{-1}$ and $h^{-1}$ by
inverting the corresponding words. We denote by $\Sigma_{[g,h]}$
the corresponding word in the entries of $\overline{\alpha}$ and
their inverses (which is equal to $[g,h]$ in $G$). Clearly, the
total degree (in $\Sigma_{[g,h]}$) of each entry $g_{i}$ is zero
and hence permuting the elements of $\Sigma_{[g,h]}$ we obtain the
identity element $e$. It follows easily that taking products of
such words we obtain a word $\Sigma = \Sigma_{z}$ whose different
permutations yield all of $G^{'}$. This would complete the proof
of the lemma if we assume that whenever $g$ is an entry in
$\overline{\alpha}$ then also $g^{-1}$ is. But clearly, if this is
not the case, we can replace $g^{-1}$ by $g^{ord(g)-1}$ and the
result follows.

\end{proof}

Consider the ($r$-dimensional) lattice $\Gamma_{r}=
(\mathbb{Z}_{+})^{(r)}$ of non negative integer points, where $r$
is the cardinality of $\overline{\alpha}$. We refer to
$\Gamma_{r}$ as the $r$-dimensional non-negative Euclidean
lattice. Similarly we will consider $\Gamma_{k}$, $k \leq r$, the
$k$-dimensional non-negative Euclidean lattices and their
translations, $\overrightarrow{x} + \Gamma_{k}$ where
$\overrightarrow{x} \in (\mathbb{Z}_{+})^{(k)}$. We view the lattice
$\Gamma_{r}$, as a partial ordered set where $A = (n_1,\ldots,n_r)
\prec B=(m_1,\ldots,m_r)$ if and only if $n_i \leq m_i$ for $1
\leq i \leq r$. Clearly, this partial order inherits a partial
order on $\Gamma_{k}$, $k \leq r$, and their translations. To any
point $A = (n_1,\ldots,n_r)$, $n_i \geq 0$ in $\Gamma_{r}$ we
attach all monomials $X_{\Sigma}$ with number of variables as
prescribed by the point $A$, that is, the variable $x_{1,g_1}$
appears $n_1$ times, $x_{2,g_2}$ appears $n_2$ and so on. Clearly
any word $\Sigma$ determines a unique lattice point $A$ in which
case we write $A=A_{\Sigma}$ or $\Sigma \in A$. Clearly the
elements in $G$ represented by all monomials that correspond to a
point $A \in \Gamma_{r}$ lie in the same coset of $G^{'}$ and
hence denoting by $N_{A}= \{g \in G :$  $g$ is represented by
monomials in $A$\}, we have that $ 1 \leq \ord(N_{A}) \leq
\ord(G^{'})$.

\begin{lemma} \label{monotonicity}

The function $\ord(N(A)): \Gamma_{r} \rightarrow
\{1,\ldots,\ord(G^{'})\}$ is monotonic (increasing) with respect
to partial ordering on $\Gamma_{r}$.

In particular if $A_{\Sigma} \in \Gamma_{r}$, where $\Sigma$ is a word in the entries of $\overline{\alpha}$
whose different permutations represent all elements of $G^{'}$ (as
constructed in Lemma \ref{RICH WORDS})), then for any word $\Pi$ such that $A_{\Pi} \succeq A_{\Sigma}$ we have
$\ord(N(A_{\Pi})) = \ord(G^{'})$.

\end{lemma}

\begin{proof} This is clear.
\end{proof}

We can now complete the proof that the Hilbert series of the
corresponding relatively free algebra is rational.

Let $A = (n_1,\ldots,n_r) \in \Gamma_{r}$, where $n= n_1 + \cdots + n_r$, and
$P_{A}$ be the space spanned by all monomials $X_{\Pi}$ where
$\Pi\in A$. Since the $T$-ideal of identities is spanned by
strongly homogeneous polynomials we have that the subspace of the
relatively free algebra spanned by monomials of degree $n$ is
decomposed into the direct sum of spaces which are spanned by
monomials in lattice points $A$ of degree $n$.

We claim that for any integer $\lambda$, $1\leq \lambda\leq \ord(G^{'})$, the set of points $A \in \Gamma_{r}$ such that

$$
\dim(P_{A}/P_{A} \cap (\id_{G}(FG)))=\lambda
$$
is a \textit{finite} union of \textit{disjoint} sets which are
translations of lattices $\Gamma_{k}$, $0 \leq k \leq r$. We
present here a proof which was shown to us by Uri Bader. Consider
the one point compactification $\widehat{\mathbb{Z}_{+}}$ of
$\mathbb{Z}_{+}$. It is convenient to view the space
$\widehat{\mathbb{Z}_{+}}$ as homeomorphic to the set of points
$I_{\mathbb{N}}=\{1/n: n \in \mathbb{N}\} \cup \{0\}$ with the
induced topology of the Lesbegue measure on the interval $[0,1]$.
The closed sets are either the finite sets or those that contain
$0$. Consequently the sets which are closed and open are either
the finite sets without $0$ or sets that contain a set of the form
$\{1/n: n \geq d\} \cup \{0\}$. Consider the $r$-fold cartesian
product

$$
I^{r}_{\mathbb{N}}=I_{\mathbb{N}} \times I_{\mathbb{N}} \times \cdots \times I_{\mathbb{N}}
$$

Clearly $I^{r}_{\mathbb{N}}$ is compact. Furthermore the function
$\ord(N(A))$ with values in the finite set $T=\{1,\ldots,
\ord(G^{'})\}$ (viewed as a function on $I^{r}_{\mathbb{N}}$) is
monotonic decreasing and hence continuous. It follows that the
inverse image of any point in $T$ is an open and closed subset of
$I^{r}_{\mathbb{N}}$ and the result now follows easily.

Returning to $\Gamma_{r}= (\mathbb{Z}_{+})^{(r)}$ we see
that the rationality of the series will follow if we know the
rationality of the Hilbert series which corresponds to the
enumeration of lattice points of degree $n$ in  $\Gamma_{k}$, $k \leq r$.
Thus we need to check the rationality of the following
power series

$$
\sum_{n} ((n+k-1)!/(n!)(k-1)!) t^{n}
$$
and this is clear. This gives a direct proof of Theorem \ref{main basic theorem} in case $I=\id_{G}(FG)$.

\bigskip

We close the article with an estimate of the sequence of
codimensions which corresponds to the $T$-ideal $I=\id_{G}(FG)$.
For the calculation we consider the $F$-space $P_{n}$ spanned by
all $G$-graded multilinear monomials of degree $n$ on the
variables $\{x_{g,i}\}_{{g \in G},{i=1,\ldots, n}}$ which are
permutations of monomials of the form $x_{1,g_1}x_{2,g_2}\ldots
x_{n,g_n}$, where $(g_1,\ldots,g_n) \in G^{(n)}$. Clearly,
$\dim(P_{n})= ord (G)^{n} \times n!$. Let

$$
c^{n}_{G}(FG)=P_{n}/(P_{n}\cap I).
$$

The integer $c^{n}_{G}(FG)$ is the $n$-codimension which corresponds to the $G$-graded $T$-ideal $I$.

\begin{corollary} \label{codimension for FG}

Let $A=FG$ be the group algebra over a field $F$ and let
$\id_{G}(FG)$ be the $G$-graded $T$-ideal of identities. Let
$c^{n}_{G}(FG)$ be the $n$-th coefficient of the codimension
sequence of $\id_{G}(FG)$. Then for any integer $n$ we have

\begin{enumerate}

\item
$$
ord (G)^{n} \leq c^{n}_{G}(FG) \leq \ord (G^{'})\ord (G)^{n}.
$$

\item

$$
\lim_{n\rightarrow \infty} (c^{n}_{G}(FG)/\ord (G^{'})\ord (G)^{n}))=1
$$

\end{enumerate}
In particular $\exp_{G}(FG) = \lim_{n\rightarrow \infty} \sqrt[n]{c^{n}_{G}(FG)}=\ord(G)$ $($see \cite{AGM}$)$.

\end{corollary}

\begin{proof}

Let $\textbf{x}=x_{1,g_{1}}x_{2,g_{2}} \ldots x_{n,g_{n}}$ be a
$G$-graded monomial in $P_{n}$. As noted above, the set of all
$n!$ permutations of $\textbf{x}$ decomposes into at most $G^{'}$
equivalence classes where two monomials are equivalent if and only
if the difference is a binomial identity of $FG$. It then follows
easily that $ord (G)^{n} \leq c^{n}_{G}(FG) \leq \ord (G^{'})\ord
(G)^{n}.$

For the second part, recall from the second part of Lemma \ref{monotonicity} that if a
monomial contains $G$-graded variables with degrees in $G$ as in
the word $\Sigma$ (including multiplicities), then its
permutations yield precisely $G^{'}$ non equivalent classes. So
the second part of the corollary will follow easily if we show
that

$$
\lim_{n\rightarrow \infty} \frac{d_n}{ord (G)^{n} \times n!}=0
$$
where $d_{n}$ denotes the number of monomials in $P_{n}$ which do
not contain the set of elements of $\Sigma$ (with repetitions). But this
of course follows from an easy calculation which is omitted.

\end{proof}

\begin{remark} \label{twisting}

One may replace the group algebra $FG$ above by any twisted group
algebra $F^{\alpha}G$, where $\alpha$ is a $2$-cocycle of $G$ with
values in $F^{*}$. As above, also here, two monomials are
equivalent if and only if they have the same variables and they
determine elements in the same $g$-homogeneous component. The only
difference (comparing to the case where $\alpha \equiv 1$) is that
here, for two such monomials, the polynomial identity they
determine has the form

$$x_{g_{i_1}}x_{g_{i_2}}\cdots
x_{g_{i_n}}- \gamma
x_{g_{i_{\sigma(1)}}}x_{g_{i_{\sigma(2)}}}\cdots
x_{g_{i_{\sigma(n)}}}
$$
where $\gamma$ is a non zero element of $F$ which is determined by
the $2$-cocycle $\alpha$ and the words  $g_{i_1}g_{i_2}\cdots
g_{i_n}$ and $g_{i_{\sigma(1)}}g_{i_{\sigma(1)}}\cdots
g_{i_{\sigma(1)}}$ (see \cite{AHN}). One sees easily that the
twisting by the $2$-cocycle $\alpha$ has no effect neither on the
Hilbert series nor on the sequence of codimnesions. Details are
left to the reader.

\end{remark}

\end{section}

\end{document}